\RequirePackage{etex}
\documentclass[10pt]{amsart}
\usepackage{tabularx}
\usepackage{fancyhdr}
\usepackage{float}
\usepackage[margin=1.35 in]{geometry}
\usepackage[english]{babel}
\usepackage[utf8]{inputenc}
\usepackage{subcaption}
\usepackage{amsmath}
\usepackage{amssymb}
\usepackage{amsfonts}
\usepackage{amsthm}
\usepackage{mathdots}
\usepackage{mathrsfs}
\usepackage[all]{xy}
\usepackage[pdftex]{graphicx}
\usepackage{color}
\usepackage{cite}
\usepackage{url}
\usepackage{indent first}
\usepackage[labelfont=bf,labelsep=period,justification=raggedright]{caption}
\usepackage[english]{babel}
\usepackage[utf8]{inputenc}
\usepackage[colorinlistoftodos]{todonotes}
\usepackage{tkz-fct}
\usepackage{tikz}
\usepackage{tikz-3dplot}
\usetikzlibrary{patterns}
\usetikzlibrary{calc}
\usetikzlibrary{knots}
\usetikzlibrary{spath3, intersections, hobby}
\usetikzlibrary{arrows, automata}
\usetikzlibrary{arrows.meta}
\usetikzlibrary{shapes.multipart}
\usepackage{pgfplots}
\usepackage{multicol}
\PassOptionsToPackage{dvipsnames,svgnames}{xcolor}
\usepackage{textcomp}
\usepackage{bbm}
\usepackage{array}
\newcolumntype{C}{>{$}c<{$}}
\usepackage{forest}
\usepackage{tikzsymbols}
\usetikzlibrary{3d}
\tdplotsetmaincoords{80}{110}
\usepackage{worldflags}
\usepackage{youngtab}
\usepackage{tikzducks}
\usepackage{hyperref}

\usepackage{framed}

\makeatletter
\DeclareFontFamily{OMX}{MnSymbolE}{}
\DeclareSymbolFont{MnLargeSymbols}{OMX}{MnSymbolE}{m}{n}
\SetSymbolFont{MnLargeSymbols}{bold}{OMX}{MnSymbolE}{b}{n}
\DeclareFontShape{OMX}{MnSymbolE}{m}{n}{
    <-6>  MnSymbolE5
   <6-7>  MnSymbolE6
   <7-8>  MnSymbolE7
   <8-9>  MnSymbolE8
   <9-10> MnSymbolE9
  <10-12> MnSymbolE10
  <12->   MnSymbolE12
}{}
\DeclareFontShape{OMX}{MnSymbolE}{b}{n}{
    <-6>  MnSymbolE-Bold5
   <6-7>  MnSymbolE-Bold6
   <7-8>  MnSymbolE-Bold7
   <8-9>  MnSymbolE-Bold8
   <9-10> MnSymbolE-Bold9
  <10-12> MnSymbolE-Bold10
  <12->   MnSymbolE-Bold12
}{}

\let\llangle\@undefined
\let\rrangle\@undefined
\DeclareMathDelimiter{\llangle}{\mathopen}%
                     {MnLargeSymbols}{'164}{MnLargeSymbols}{'164}
\DeclareMathDelimiter{\rrangle}{\mathclose}%
                     {MnLargeSymbols}{'171}{MnLargeSymbols}{'171}
\makeatother

\makeatletter
\tikzoption{canvas is xy plane at z}[]{%
  \def\tikz@plane@origin{\pgfpointxyz{0}{0}{#1}}%
  \def\tikz@plane@x{\pgfpointxyz{1}{0}{#1}}%
  \def\tikz@plane@y{\pgfpointxyz{0}{1}{#1}}%
  \tikz@canvas@is@plane
}
\makeatother

\tikzset{surface/.style={draw=black, left color=orange,right color=orange,middle
color=orange!60!#1, fill opacity=1},surface/.default=white}

\pgfplotsset{compat=1.17}

\newcommand{\eps}{\varepsilon}

\newcommand{\QQ}{\mathbb{Q}}

\newcommand{\ZZ}{\mathbb{Z}}

\def\multiset#1#2{\ensuremath{\left(\kern-.3em\left(\genfrac{}{}{0pt}{}{#1}{#2}\right)\kern-.3em\right)}}

\makeatletter

\newif\ifpgfcirclecrosssplitcustomfill
\tikzset{%
  circle cross split part fill/.code=\def\pgf@lib@sh@ccs@list@fill{#1}\pgfcirclecrosssplitcustomfilltrue,%
  circle cross split uses custom fill/.is if=pgfcirclecrosssplitcustomfill}
\pgfdeclareshape{circle cross split}{%
  \nodeparts{text,two,three,four}%
  \savedanchor\centerpoint{%
    \pgfmathsetlength\pgf@xa{\pgfkeysvalueof{/pgf/inner xsep}}%
    \pgfmathsetlength\pgf@ya{\pgfkeysvalueof{/pgf/inner ysep}}%
    \pgf@x\wd\pgfnodeparttextbox
    \pgf@yb\dp\pgfnodeparttextbox
    \pgf@yc\dp\pgfnodeparttwobox
    \ifdim\pgf@yb>\pgf@yc
      \pgf@yc\pgf@yb
    \fi
    \advance\pgf@y-\pgf@yc
    \advance\pgf@x\pgf@xa
    \advance\pgf@y-\pgf@ya
    \advance\pgf@x.5\pgflinewidth
    \advance\pgf@y-.5\pgflinewidth
  }%
  \savedanchor\twoanchor{%
    \pgfmathsetlength\pgf@xa{\pgfkeysvalueof{/pgf/inner xsep}}%
    \pgfmathsetlength\pgf@ya{\pgfkeysvalueof{/pgf/inner ysep}}%
    \advance\pgf@x.5\pgflinewidth
    \advance\pgf@x\pgf@xa
    \advance\pgf@y.5\pgflinewidth
    \advance\pgf@y\pgf@ya
    \pgf@yb\dp\pgfnodeparttextbox
    \pgf@yc\dp\pgfnodeparttwobox
    \ifdim\pgf@yb>\pgf@yc
      \pgf@yc\pgf@yb
    \fi
    \advance\pgf@y\pgf@yc
  }%
  \savedanchor\threeanchor{%
    \pgfmathsetlength\pgf@ya{\pgfkeysvalueof{/pgf/inner ysep}}%
    \pgf@x\wd\pgfnodeparttextbox
    \pgf@yb\dp\pgfnodeparttextbox
    \pgf@yc\dp\pgfnodeparttwobox
    \ifdim\pgf@yb>\pgf@yc
      \pgf@yc\pgf@yb
    \fi
    \advance\pgf@y-\pgf@yc
    \advance\pgf@y-2\pgf@ya
    \advance\pgf@y-\pgflinewidth
    \pgf@yb\ht\pgfnodepartthreebox
    \pgf@yc\ht\pgfnodepartfourbox
    \ifdim\pgf@yb>\pgf@yc
      \pgf@yc\pgf@yb
    \fi
    \advance\pgf@y-\pgf@yc
    \advance\pgf@x-\wd\pgfnodepartthreebox
  }%
  \savedanchor\fouranchor{%
    \pgfmathsetlength\pgf@xa{\pgfkeysvalueof{/pgf/inner xsep}}%
    \advance\pgf@x\wd\pgfnodepartthreebox
    \advance\pgf@x2\pgf@xa
    \advance\pgf@x\pgflinewidth
  }%
  \saveddimen\radius{%
    \pgf@y\ht\pgfnodeparttextbox
    \pgf@yb\ht\pgfnodeparttwobox
    \ifdim\pgf@yb>\pgf@y
      \pgf@y\pgf@yb
    \fi
    \pgf@yc\dp\pgfnodeparttextbox
    \pgf@yb\dp\pgfnodeparttwobox
    \ifdim\pgf@yc>\pgf@yb
      \advance\pgf@y\pgf@yc
    \else
      \advance\pgf@y\pgf@yb
    \fi
    \pgf@yb\ht\pgfnodepartthreebox
    \ifdim\pgf@yb<\ht\pgfnodepartfourbox
      \pgf@yb\ht\pgfnodepartfourbox
    \fi
    \pgf@yc\dp\pgfnodepartthreebox
    \ifdim\pgf@yc<\dp\pgfnodepartfourbox
      \advance\pgf@yb\dp\pgfnodepartfourbox
    \else
      \advance\pgf@yb\pgf@yc
    \fi
    \ifdim\pgf@yc>\pgf@y
      \pgf@y\pgf@yc
    \fi
    \pgfmathsetlength\pgf@ya{\pgfkeysvalueof{/pgf/inner ysep}}%
    \advance\pgf@y2\pgf@ya
    \pgf@x\wd\pgfnodeparttextbox
    \pgf@xa\wd\pgfnodepartthreebox
    \pgf@xb\wd\pgfnodeparttwobox
    \pgf@xc\wd\pgfnodepartfourbox
    \ifdim\pgf@xa>\pgf@x
      \pgf@x\pgf@xa
    \fi
    \ifdim\pgf@xb>\pgf@x
      \pgf@x\pgf@xb
    \fi
    \ifdim\pgf@xc>\pgf@x
      \pgf@x\pgf@xc
    \fi
    \pgfmathsetlength\pgf@xa{\pgfkeysvalueof{/pgf/inner xsep}}%
    \advance\pgf@x2\pgf@xa
    \ifdim\pgf@y>\pgf@x
      \pgf@x\pgf@y
    \fi
    \advance\pgf@x.5\pgflinewidth
    \pgfmathsetlength{\pgf@xb}{\pgfkeysvalueof{/pgf/minimum width}}%
    \pgfmathsetlength{\pgf@yb}{\pgfkeysvalueof{/pgf/minimum height}}%
    \ifdim\pgf@x<.5\pgf@xb
        \pgf@x=.5\pgf@xb
    \fi
    \ifdim\pgf@x<.5\pgf@yb
        \pgf@x=.5\pgf@yb
    \fi
    \pgfmathsetlength{\pgf@xb}{\pgfkeysvalueof{/pgf/outer xsep}}%
    \pgfmathsetlength{\pgf@yb}{\pgfkeysvalueof{/pgf/outer ysep}}%
    \ifdim\pgf@xb<\pgf@yb
      \advance\pgf@x\pgf@yb
    \else
      \advance\pgf@x\pgf@xb
    \fi
  }%
  \inheritanchorborder[from=circle]%
  \inheritanchor[from=circle]{north}%
  \inheritanchor[from=circle]{north west}%
  \inheritanchor[from=circle]{north east}%
  \inheritanchor[from=circle]{center}%
  \inheritanchor[from=circle]{west}%
  \inheritanchor[from=circle]{east}%
  \inheritanchor[from=circle]{mid}%
  \inheritanchor[from=circle]{mid west}%
  \inheritanchor[from=circle]{mid east}%
  \inheritanchor[from=circle]{base}%
  \inheritanchor[from=circle]{base west}%
  \inheritanchor[from=circle]{base east}%
  \inheritanchor[from=circle]{south}%
  \inheritanchor[from=circle]{south west}%
  \inheritanchor[from=circle]{south east}%
  \anchor{two}{\twoanchor}%
  \anchor{three}{\threeanchor}%
  \anchor{four}{\fouranchor}%
  \inheritbackgroundpath[from=circle]
  \beforebackgroundpath{%
    \pgfutil@tempdima=\radius
    \pgfmathsetlength{\pgf@xb}{\pgfkeysvalueof{/pgf/outer xsep}}%
    \pgfmathsetlength{\pgf@yb}{\pgfkeysvalueof{/pgf/outer ysep}}%
    \ifdim\pgf@xb<\pgf@yb
      \advance\pgfutil@tempdima by-\pgf@yb
    \else
      \advance\pgfutil@tempdima by-\pgf@xb
    \fi
    \advance\pgfutil@tempdima by-.5\pgflinewidth%
    \pgfsetshortenstart{0pt}%
    \pgfsetshortenend{0pt}%
    \pgfsetarrows{-}%
    \pgfpathmoveto{\pgfpointadd{\centerpoint}{\pgfqpoint{-\pgfutil@tempdima}{0pt}}}%
    \pgfpathlineto{\pgfpointadd{\centerpoint}{\pgfqpoint{\pgfutil@tempdima}{0pt}}}%
    \pgfpathmoveto{\pgfpointadd{\centerpoint}{\pgfqpoint{0pt}{-\pgfutil@tempdima}}}%
    \pgfpathlineto{\pgfpointadd{\centerpoint}{\pgfqpoint{0pt}{\pgfutil@tempdima}}}%
    \pgfusepathqstroke
  }%
  \behindbackgroundpath{%
    \pgfutil@tempdima=\radius
    \pgfmathsetlength{\pgf@xb}{\pgfkeysvalueof{/pgf/outer xsep}}%
    \pgfmathsetlength{\pgf@yb}{\pgfkeysvalueof{/pgf/outer ysep}}%
    \ifdim\pgf@xb<\pgf@yb
      \advance\pgfutil@tempdima by-\pgf@yb
    \else
      \advance\pgfutil@tempdima by-\pgf@xb
    \fi
    \advance\pgfutil@tempdima by-.5\pgflinewidth%
    \ifpgfcirclecrosssplitcustomfill%
      \pgf@lib@sh@rs@process@list{\pgf@lib@sh@ccs@list@fill}{4}%
      {%
        \pgfmathloop
           \ifnum\pgfmathcounter>4%
           \else%
             \pgf@lib@sh@getalpha\pgf@lib@sh@rs@number{\pgfmathcounter}%
              \edef\pgf@tempa{\csname pgf@lib@sh@rs@\pgf@lib@sh@rs@number @item\endcsname}%
              \ifx\pgf@tempa\pgf@lib@sh@rs@nonetext\else
                \pgfsetfillcolor{\pgf@tempa}%
                \pgf@lib@sh@ccs@angles{\pgfmathcounter}%
                \pgfpathmoveto{\centerpoint}%
                \pgfpathlineto{\pgfpointadd{\centerpoint}{\pgfqpointpolar{\pgf@lib@sh@ccs@angle}{\pgfutil@tempdima}}}%
                \pgfpatharc{\pgf@lib@sh@ccs@angle}{\pgf@lib@sh@ccs@angle@}{\pgfutil@tempdima}%
                \pgfpathclose
                \pgfusepathqfill
              \fi
        \repeatpgfmathloop
      }%
    \fi%
  }%
}
\def\pgf@lib@sh@ccs@angles#1{%
  \ifcase#1\or\def\pgf@lib@sh@ccs@angle{90}%
           \or\def\pgf@lib@sh@ccs@angle{0}%
           \or\def\pgf@lib@sh@ccs@angle{180}%
           \else\def\pgf@lib@sh@ccs@angle{270}%
  \fi
  \edef\pgf@lib@sh@ccs@angle@{\number\numexpr\pgf@lib@sh@ccs@angle+90\relax}%
}
\makeatother

\theoremstyle{plain}
\newtheorem{thm}{Theorem}
\newtheorem{lemma}[thm]{Lemma}
\newtheorem{cor}[thm]{Corollary}

\newtheorem{prop}[thm]{Proposition}

\theoremstyle{definition}
\newtheorem{defn}{Definition}

\theoremstyle{remark}
\newtheorem{rem}{Remark}
\newtheorem{ex}{Example}

\numberwithin{equation}{section}
\usepackage{cleveref}

\title{The Quadratic and Cubic Characters of 2}
\date{June 2025}
\author{Matias C. Relyea}

\begin{document}

\maketitle

\begin{abstract}
      The solvability of the cubic congruence $x^{3}\equiv 2\pmod{p}$ is referred to as the \textit{cubic character of 2}. In evaluating the cubic character of 2, we introduce the Eisenstein integers, Gauss and Jacobi sums, and the law of cubic reciprocity. We motivate this proof by giving ample historical information surrounding the early development of higher reciprocity laws as well as Gauss' proof of the solvability of the quadratic congruence $x^{2}\equiv 2\pmod{p}$; conventionally the \textit{quadratic character of 2}. We simultaneously outline other relevant contributions by Fermat, Euler, Legendre, Jacobi, and Eisenstein. 
\end{abstract}

\section*{From the Beginning}
Many elementary number theory texts cover everything up to quadratic reciprocity and rarely anything further pertaining to reciprocity laws. While it is true that higher reciprocity laws such as cubic, biquadratic, Eisenstein, or even Artin reciprocity are rooted in mechanics demanding a sufficient amount of algebraic number theory, their special cases are approachable with only a few additional ideas from algebra. In a similar way, the history of reciprocity laws from their origin is undoubtedly rich, hence Lemmermeyer’s book \cite{lemmermeyer2000reciprocity}. It is illuminating to understand this rich history before encountering a modern treatment of reciprocity laws. In this paper, we are concerned with relating the solvability of the cubic congruence $x^{3}\equiv 2\pmod{p}$ to an elementary representation of primes while utilizing some fundamental properties of Gauss and Jacobi sums. We will also use Gauss’ approach to solving the quadratic congruence $x^{2}\equiv 2\pmod{p}$ as the foundation for how our story develops. We do this while considering the nuance of historical contributions from Fermat, Euler, Gauss, Legendre, Jacobi, and Eisenstein. 

\indent The law of quadratic reciprocity is perhaps the most well-known theorem of elementary number theory, and it is typically the first major result encountered by most students in a first course. Many refer to the law as the pinnacle of elementary number theory; even Carl Friedrich Gauss, who gave the first proofs of the law, referred to it as the ``Theorema Aureum," or the ``Golden Theorem." The main theorem is as follows, where $(a/p)$ denotes the Legendre symbol\footnote{Recall that the Legendre symbol $(a/p)$ is defined to be $1$ when $a$ is a solution to the quadratic congruence $x^{2}\equiv a\pmod{p}$ (a quadratic residue), and $-1$ otherwise ($a$ a quadratic nonresidue). If $a\equiv 0\pmod{p}$, then $(a/p)$ is defined to be $0$.}.
\begin{thm}[The Law of Quadratic Reciprocity] \label{the law of quadratic reciprocity}
    Let $p,q\in\mathbb{Z}$ be odd primes. Then
        \begin{equation*}
            (p/q)(q/p) =(-1)^{\frac{p-1}{2}\cdot\frac{q-1}{2}}.
        \end{equation*}
\end{thm}
While there exist multiple forms of the law, this form, which was first expressed by Adrien-Marie Legendre, is certainly the most elegant. Gauss proved this magnificent result a total of 8 times. When the law was initially stated, there were also two supplementary laws. In this paper, we are most interested in the second.
\begin{thm}[Second Supplement to Theorem \ref{the law of quadratic reciprocity}] \label{supplement to quadratic reciprocity}
    Let $p$ be an odd prime. Then
        \begin{equation*}
            (2/p) = (-1)^{\frac{p^{2}-1}{8}}.
        \end{equation*}
\end{thm}
This theorem concerns the solvability of the quadratic congruence $x^{2}\equiv 2\pmod{p}$, and its proof is infrequently touched upon in elementary number theory classes. The solvability of this congruence is referred to as the \textit{quadratic character of 2}; the reason for this language will be explained later. It is one of the two main results that we are interested in proving in a part of this paper\footnote{Interestingly, problem number 4 on the 2023 Moroccan Baccalaureate, a graduation exam given each year to hundreds of thousands of Moroccan students at the end of secondary school, asked students to provide a proof of this result. Remarkably, there are numerous elementary proofs of this result, most following a combinatorial argument.}.

\indent Gauss determines the quadratic character of $2$ in \textit{Disquisitiones Arithmeticae} \cite{Gauß_1986} and proves it using induction. While the conventional proof seen in most number theory texts uses Gauss' Lemma and a combinatorial argument, Gauss' inductive proof informs us exactly how he approached proving numerous other laws concerning residues and nonresidues (most notably, quadratic reciprocity). We give Gauss' proof from \textit{Disquisitiones} in the next section.

\indent Even fewer introductory university classes cover special cases of higher reciprocity. Nonetheless, we can still use elementary results to get a taste of the consequences of cubic reciprocity by considering the solvability of the special cubic congruence $x^{3}\equiv 2\pmod{p}$. The solvability of this congruence is referred to as the \textit{cubic character of 2}. After proving quadratic reciprocity, Gauss took an interest in higher reciprocity. In particular, he closely studied the solvability of congruences of the form $x^{3}\equiv a\pmod{p}$ and $x^{4}\equiv a\pmod{p}$; respectively, these are known as cubic and biquadratic reciprocity. Despite Gauss' fascination with these congruences, his theorems on cubic and biquadratic residues were only conjectures. In his second monograph in which he also considered biquadratic reciprocity, Gauss stated in a footnote—his only published reference to cubic residues—that

\begin{quote}
    [t]he theory of cubic residues must be based in a similar way on a consideration of numbers of the form $a + bh$ where h is an imaginary root of the equation $h^{3}-1 = 0$, say $h=(-1+\sqrt{-3})/2$, and similarly the theory of residues of higher powers leads to the introduction of other imaginary quantities \cite{biblioteca1863werke2} (translated \cite[pg. 84]{cox2013primes}).
\end{quote}

As we know it today, the value $``h"$ that Gauss was referring to is $\omega$, a cube root of unity, and the set of numbers of the form $a+bh$ is now known as the Eisenstein integers. Gauss' speculation came to no fruition, and it wasn't until 1844 that cubic reciprocity was first proven by Gotthold Eisenstein, who also happened to be one of Gauss' students. There is, however, some debate regarding whether Gauss' notes containing a proof of the law—which eventually appeared in posthumous monographs—were written before or after Eisenstein published his manuscripts. In June of 1844, after reading Eisenstein's manuscripts on cubic reciprocity, Gauss even invited Eisenstein to his home in Göttingen to discuss them. The proof that Eisenstein produced used primary numbers and the residue class ring $\ZZ[\omega]/\pi\ZZ[\omega]$, where $\pi$ is some prime element of $\ZZ[\omega]$. We will discuss the importance of primary numbers later when we delve into cubic reciprocity more deeply, but fundamentally they were created by Eisenstein to differentiate between particular elements of the Eisenstein integers with respect to unique factorization. Eisenstein's proof also implements cubic Gauss sums, which Gauss had developed several years earlier. 

\indent In perfect analogy with the Legendre symbol, we let $(\alpha/\pi)_{3}$ be the cubic character of $\alpha$ for some prime element $\pi$ and $\alpha$ an arbitrary element of $\ZZ[\omega]$. The law of cubic reciprocity can be stated as follows.

\begin{thm}[The Law of Cubic Reciprocity] \label{The Law of Cubic Reciprocity}
    Let $\pi_{1}$ and $\pi_{2}$ be primary elements of $\ZZ[\omega]$. Furthermore, let $N\pi_{1},N\pi_{2}\neq 3$ and $N\pi_{1}\neq N\pi_{2}$, or unequal norms\footnote{The case when $N\pi_{1}=N\pi_{2}$ is an interesting nontrivial problem relating closely to cubic diophantine equations, but it lies outside the scope of this paper.}. Then 
    \begin{equation*}
        (\pi_{1}/\pi_{2})_{3} = (\pi_{2}/\pi_{1})_{3}.
    \end{equation*}
\end{thm}

The proof of cubic reciprocity is by no means simple and may be found completely in \cite{relyea2024finitefieldshigherreciprocity} as well as in chapter 9 of \cite{ireland1990classical}.

\indent While the main supplements of cubic reciprocity concern prime elements of the Eisenstein integers, the case of $2$ is still incredibly nuanced. Concerning the cubic character of 2, Gauss wrote, in his posthumously published \textit{Werke VIII}, that 

\begin{quote}
    2 is a cubic residue or nonresidue of a prime number $p$ of the form $3n+1$, according to whether $p$ is representable by the form $xx+27yy$ or $4xx+2xy+7yy$ \cite{gauss1900werke8} (translated \cite[pg. 85]{cox2013primes}).
\end{quote}

The form $xx+27yy$ is very similar to the form of primes in a result concerning the cubic character of 2 conjectured by Leonhard Euler between 1748 and 1750. At the time, a well known result of Pierre de Fermat—a consequence of his theorem on the sum of two squares—was that if $p\equiv 1\pmod{3}$, then for $a,b\in\ZZ$, the representation $p=a^{2}+3b^{2}$ was unique up to sign. Using intuition from Fermat's theorem, Euler stated, among other cubic characters of $3,5,6,7,$ and $10$, that $(2/p)=1$ if and only if $3|b$, where $p$ was in the aforementioned representation $a^{2}+3b^{2}$ (see chapter 7 of \cite{lemmermeyer2000reciprocity}). This result is the main theorem of this paper, and its formal proof is laid out in the final section. The main theorem is as follows.
\begin{thm} \label{Theorem: MAIN THEOREM}
    If $p\equiv 1\pmod{3}$, then the cubic congruence $x^{3}\equiv 2\pmod{p}$ is solvable if and only if there exist integers $C$ and $D$ such that $p=C^{2}+27D^{2}$. 
\end{thm}

While this result \textit{was} conjectured by Euler as a result of his ``genius," he did not provide a proof. Even though Theorem \ref{Theorem: MAIN THEOREM} was eventually proven by Gauss in posthumously published notes, this paper aims to illuminate the historical nuance of contributions from other mathematicians leading up to Gauss' work on the cubic character of 2, and more generally on cubic and biquadratic reciprocity. 

\section*{The Quadratic Character of 2}
Let us take a brief look at Gauss' evaluation of the quadratic character of 2. 

\indent Notice that Theorem \ref{supplement to quadratic reciprocity} alternatively states that if $p\equiv \pm 1\pmod{8}$, then $2$ is a quadratic residue modulo $p$, and if $p\equiv \pm 3\pmod{8}$, then $2$ is a quadratic nonresidue. Gauss' statement of this result in \textit{Disquisitiones} divides the entire result into multiple articles, but it is unsurprising considering this was how he addressed numerous problems. In condensed form, Gauss states the following.
\begin{thm}[Gauss' equivalent statement of Theorem \ref{supplement to quadratic reciprocity} \cite{Gauß_1986} \label{Proposition: Gauss form of quad char of 2}]
Let $n\in\ZZ$.
    \begin{enumerate}
        \item $+2$ will be a nonresidue, $-2$ a residue of all prime numbers of the form $8n+3$,
        \item $+2$ and $-2$ will both be nonresidues of all prime numbers of the form $8n+5$,
        \item $-2$ is a nonresidue, $+2$ a residue of all prime numbers of the form $8n+7$,
        \item $2$ and $-2$ are residues of all prime numbers of the form $8n+1$. 
    \end{enumerate}
\end{thm}

The concept of Gauss' proof is, as Gauss states, by ``induction." However, it is more accurate to describe his proof as by strong induction with a flavor of contradiction. 

\begin{proof}[Proof of Theorem \ref{Proposition: Gauss form of quad char of 2}]
    Let $n\in\ZZ$ in the following. We will prove each part individually. 
    
    \indent We prove (1) and (2) first. The first thing to notice is that if a composite number is of the form either $8n+3$ or $8n+5$, then it must involve some prime factor that is of either the form $8n+3$ or $8n+5$. Otherwise, the composite numbers would instead be in the form $8n+1$ or $8n+7$. Gauss states that there is no number of either form less than $100$ such that $2$ is a residue. The choice of $100$ is arbitrary, and we could have just as easily chosen $97$ or $235$. Now we suppose that there are numbers of this form such that $2$ is a residue, and we let the least be $t$. This $t$ must be of the form $8n+3$ or $8n+5$. Clearly, $2$ will be a residue of $t$ but a nonresidue of everything less than $t$. By definition, $2\equiv a^{2}\pmod{t}$. There must exist an odd $a<t$ such that this is true. Rearrange, so $a^{2}=2+tu$ for some integer $u$. Then $tu=a^{2}-2$. Notice that $a^{2}$ is in the form $8n+1$. Then $tu$ is of the form $=8n+1-2=8n-1$. Thus $u$ will be of the form $8n+3$ or $8n+5$, depending on whether $t$ is in the form $8n+5$ or $8n+3$ respectively. However, $a^{2}=2+tu$ also implies that $2\equiv a^{2}\pmod{u}$, or $2$ is a residue of $u$. Clearly, $u<t$, so there is a contradiction. 
    
    \indent The proof of (3) is similar. Every composite number of the form $8n+5$ or $8n+7$ must involve a prime factor of either form, so $-2$ cannot be a residue of a number of the form $8n+5$ or $8n+7$. Suppose on the contrary that some did exist, with $t$ as the least. We repeat the procedure and arrive at the congruence $-2=a^{2}-tu$. If $a<t$ and odd, then $u$ will be of the form $8n+5$ or $8n+7$, depending on whether $t$ is of the form $8n+7$ or $8n+5$ respectively. However, $a^{2}+2=tu$ implies that $-2\equiv a^{2}\pmod{u}$, or $-2$ is a residue of $u$. Clearly, $u<t$, so there is a contradiction. 
    
    \indent The case for (4) is trickier. Instead of using the strong induction approach, Gauss makes a clever deduction using congruences. We begin by letting $8n+1$ be some prime and $a$ be a primitive root modulo $8n+1$. Since $a$ is a primitive root modulo $8n+1$, we know that $a^{(8n+1)-1}=a^{8n}\equiv 1\pmod{8n+1}$. We can factor the left-hand-side as $(a^{4n})^{2}\equiv 1\pmod{8n+1}$. This is only possible if $a^{4n}\equiv \pm 1\pmod{8n+1}$, but we take the negative part. Therefore $a^{4n}\equiv -1\pmod{8n+1}$, or $a^{4n}+1\equiv 0\pmod{8n+1}$. Adding $2a^{2n}$ to both sides, we obtain $a^{4n}+2a^{2n}+1\equiv 0\pmod{8n+1}$, and after factoring we obtain $(a^{2n}+1)^{2}\equiv 2a^{2n}\pmod{8n+1}$. Alternatively, we subtract $2a^{2n}$ from both sides. After factoring we obtain $(a^{2n}-1)^{2}\equiv -2a^{2n}\pmod{8n+1}$. In either case, we have that $\pm 2a^{2n}$ are both quadratic residues. In other words,
    \begin{equation*}
        \left( \frac{\pm 2a^{2n}}{8n+1} \right) = \left( \frac{\pm 2}{8n+1} \right)\left( \frac{a^{2n}}{8n+1} \right) = \left( \frac{\pm 2}{8n+1} \right)\left( \frac{(a^{n})^{2}}{8n+1} \right) =  \left( \frac{\pm 2}{8n+1} \right) = 1.
    \end{equation*}
    Therefore both $2$ and $-2$ are quadratic residues modulo $8n+1$. 
\end{proof}
It is worth noting that Gauss didn't use Legendre symbols, as they were not introduced until 1798 when Legendre first formally defined them\footnote{While Gauss managed to prove the quadratic character of 2 using induction, Legendre was focused on the quadratic character of 2 using techniques stemming from quadratic forms, a technique that Gauss succeeded in implementing in his proof of quadratic reciprocity in \textit{Disquisitiones}. Legendre's proof differs drastically from Gauss', but it is interesting nonetheless (see \cite{collison1977cubbi} for a detailed overview).}. 
\begin{ex}
    Say we wish to determine whether there are solutions to the quadratic congruence $x^{2}\equiv 2\pmod{11}$. We can evaluate this as $(2/11)=(-1)^{\frac{11^{2}-1}{8}}=(-1)^{15}=-1$, so there is no solution. Equivalently, if we instead implement Theorem \ref{Proposition: Gauss form of quad char of 2}, we see that $11$ is of the form $8(1)+3$, so $2$ must be a nonresidue. 
\end{ex}

\section*{Eisenstein Integers, Gauss and Jacobi Sums, and the Uniqueness of a Representation of Primes}
Gauss also had considerations for the quadratic characters of other special integers in \textit{Disquisitiones}, most using similar ideas of induction. Overall, these proofs cannot be described as anything more than elementary. Yet, these investigations were quintessential to the expansion of number theory in the early 19th century. 

\indent Gauss had already given various proofs for the quadratic characters of special integers and proved quadratic reciprocity in two different ways—one by induction and the other by quadratic forms, yet—he was unsatisfied and continued to pursue more proofs. In his second memoir, Gauss stated that he 

\begin{quote}
    ... sought to add more and more proofs of the already-known theorems on quadratic residues, in the hope that from these many different methods, one or another could illuminate something in the related circumstances \cite{biblioteca1863werke2} (translated \cite[pg. 333]{Gray_2018}).
\end{quote}

At the time, Gauss was thinking extensively about cubic and biquadratic reciprocity and believed that the key to confronting these ``mysteries of the higher arithmetic" was intimately related to quadratic reciprocity. Gauss presented his 6th proof of quadratic reciprocity in 1807, and it used a new technique. In Gauss' words, 

\begin{quote}
    ... the sixth proof calls upon a completely different and most subtle principle, and gives a new example of the wonderful connection between arithmetic truths that at first glance seem to lie very far from one another \cite{biblioteca1863werke2} (translated \cite[pg. 333]{Gray_2018}).
\end{quote}

In writing this proof, Gauss defined the notion of a quadratic Gauss sum\footnote{Chapter 6 of Ireland and Rosen's book \cite{ireland1990classical} brilliantly provides a modern proof of quadratic reciprocity using quadratic Gauss sums. It turns out that quadratic Gauss sums are so powerful that they also offer an incredibly comprehensive proof of the quadratic character of 2, requiring only some theorems regarding algebraic integers. A proof can be found in the same chapter.}. While Gauss was ultimately unable to apply his newfound quadratic ``Gauss" sum to proofs of cubic and biquadratic reciprocity, his formulations of the Gauss sum were used by many of his colleagues and students, most notably by Eisenstein in his successful attempts at proving higher reciprocity laws.

\indent The idea surrounding Gauss sums is surprisingly difficult to motivate. In fact, they are so difficult to motivate within the scope of this paper that we only offer some insight into how they were developed, but in messing with the technical details, we will take its conceptual foundations for granted. In this section, we are largely concerned with proving the uniqueness of primes of the form described in Theorem \ref{Theorem: MAIN THEOREM}. On the way, we will take a look at Gauss sums and Jacobi sums and some of their fundamental properties. 

\indent As indicated by Gauss, and as mentioned earlier in \cite{biblioteca1863werke2}, considerations of cubic congruences would likely involve complex numbers with a cube root of unity. Before we prove the elegant properties of Gauss sums, we are first inclined to formally investigate the Eisenstein integers. 

\subsection*{The Eisenstein Integers}
The most significant difference between quadratic reciprocity and higher reciprocity is the use of complex numbers. Whereas quadratic reciprocity can be expressed in elementary terms over $\ZZ$, higher reciprocity (specifically Eisenstein reciprocity) is expressed over the $m$th cyclotomic field $\QQ[\zeta_{m}]$ for an integer $m>1$. The focus of this paper is on cubic reciprocity, which takes arguments from the ring $\ZZ[\omega]$, the Eisenstein integers. Every Eisenstein integer can be expressed in the form $a+b\omega$, where $a$ and $b$ are integers and $\omega=\zeta_{3}$ is a cube root of unity. We take $\omega$ to be $-1/2+i\sqrt{3}/2$.

\indent The most important property of the ring $\ZZ[\omega]$ is that it forms a unique factorization domain, or UFD. It is possible to show that $\ZZ[\omega]$ is also a Euclidean domain, which implies that there need exist a norm function over $\ZZ[\omega]$ that maps elements of $\ZZ[\omega]$ to elements of $\ZZ$. Let $\alpha=a+b\omega$. It happens that the norm over $\ZZ[\omega]$ is defined such that $N\alpha = \alpha\overline{\alpha} = (a+b\omega)(a+b\omega^{2}) = a^{2}-ab+b^{2}$, and we will prove that this norm is uniquely expressible for some $a$ and $b$ in Proposition \ref{Proposition: exist integers a and b st a^2-ab+b^2}.

\indent Necessarily, the ring $\ZZ[\omega]$ contains prime and unit elements. To determine the unit elements of $\ZZ[\omega]$, we let $a+b\omega$ be a unit. Therefore, its norm must be $1$, so determining the units amounts to determining the pairs of values of $a$ and $b$ such that $N(a+b\omega) = a^{2}-ab+b^{2}=1$. It is important to note that $\ZZ[\omega]$ also contains $\ZZ$, so to distinguish between elements of each set, we say that prime elements of $\ZZ$ are ``rational primes" and prime elements of $\ZZ[\omega]$ are just ``primes"\footnote{It is worth noting an important identity of cube roots of unity: $1+\omega+\omega^{2} = 0$. \label{Footnote: sum of cube roots of unity is 0}}.

\indent Now that we have a grasp of the basics of the Eisenstein integers, we will redirect ourselves toward considering some results concerning general Gauss and Jacobi sums. These results are technical and time-consuming, but most results are analogues to results concerning quadratic Gauss sums and follow general yet still similar proofs. Before we may take a look, however, we need to define a special symbol that characterizes (no pun intended) all residue symbols. 

\subsection*{Multiplicative Characters}
We are already familiar with the quadratic residue symbol, or the Legendre symbol $(a/p)$. This symbol behaves in a simple way: you choose some integer $a$, and you map it to either $-1,0$, or $1$. In this way, the Legendre symbol is a map between the integers and the set $\{ -1,0,1 \}$. In the case of $n=3$, the inputs can now be Eisenstein integers, with outputs as $0$ and the cube roots of unity, or $\{ 0,1,\omega,\omega^{2} \}$. As we continue increasing $n$, the domain and codomain expand even further. 

Now is a good time to mention finite fields, most notably the multiplicative group of a finite field. There are many interesting results regarding finite fields (see \cite{relyea2024finitefieldshigherreciprocity}), but the result that concerns us the most is

\begin{thm} \label{Theorem: the multiplicative group of a finite field is cyclic}
    The multiplicative group of a finite field is cyclic.
\end{thm}
This result is the building block for any construction using finite fields that we have. The conventional proof follows by Möbius inversion. This theorem also has other far-reaching consequences, both elementary and complex\footnote{A good elementary application of this result is that $(\ZZ/p\ZZ)^{\times}$ is cyclic. It is well known that $\ZZ/p\ZZ$ is a finite field. Since $\ZZ/p\ZZ$ is a finite field, by Theorem \ref{Theorem: the multiplicative group of a finite field is cyclic} its multiplicative group, $(\ZZ/p\ZZ)^{\times}$, is cyclic with order $p-1$. Using this definition, Fermat's Little Theorem also becomes a direct corollary.\label{Footnote: mult group of int mod p is cyclic}}.

In general, the map described earlier can be defined as some mapping from the multiplicative group of a finite field to specific nonzero complex numbers. This is what we meant by ``character." More formally,
\begin{defn}[Multiplicative character]
    We define the \textit{multiplicative character} on a finite field $\mathbb{F}_{q}$ to be a map $\chi$ from the multiplicative group $\mathbb{F}_{q}^{\times}$ to the nonzero complex numbers such that for all $a,b\in \mathbb{F}_{q}^{\times}$,
    \begin{equation*}
        \chi(ab)=\chi(a)\chi(b).
    \end{equation*}
\end{defn}

An important remark regarding the multiplicative character is needed.

\begin{rem} \label{Remark: character is a cyclic group}
    A character is a group homomorphism. An important fact is that multiplicative characters form a group with an identity character, $\eps$, that maps all elements to the multiplicative identity. We refer to this character as the trivial character, and it satisfies $\eps(a)=1$ for all $a\in\mathbb{F}_{q}^{\times}$. Even more surprisingly, the group of characters is cyclic, a result that is worth convincing oneself of. 
\end{rem}

Multiplicative characters are the building blocks for reciprocity laws. Though applications of the multiplicative character are more prevalent in other areas of number theory, they are particularly interesting in reciprocity laws because of how well they describe residue symbols—in turn, this lets us construct some really interesting theory. We will start by looking at some general properties of the multiplicative character.

\indent Throughout the rest of this paper, we take $q=p^{1}$ a prime, so the finite field that we refer to is $\ZZ/p\ZZ$, but we write it as $\mathbb{F}_{p}$ and denote its multiplicative group as $\mathbb{F}_{p}^{\times}$. Since $\chi$ is a map that takes in elements of $\mathbb{F}_{p}^{\times}$, we are interested in what happens if we put in different elements $a\in\mathbb{F}_{p}^{\times}$. First and foremost, we would hope that a multiplicative character maps the identity to itself. In fact, we can write $\chi(1) = \chi(1\cdot 1) = \chi(1)\chi(1)$, so the only possible value of $\chi(1)$ is $1$, which indeed maps the identity to itself. As for any element $a$, it turns out that $\chi(a)$ is just a $(p-1)$st root of unity. This is because $a^{p-1}=1$, so $1=\chi(1)=\chi(a^{p-1})=(\chi(a))^{p-1}$. From these two facts, it is also possible to show that $\chi(a^{-1})=(\chi(a))^{-1}=\overline{\chi(a)}$.

\indent An interesting fact about the Legendre symbol is that the sum of all Legendre symbols with arguments ranging from $0$ to $p-1$ is $0$. Perhaps a little unsurprisingly, the general multiplicative character satisfies the same property, as we see in the following.

\begin{prop} \label{Proposition: proof about sum of elements evaluated by a character prop}
    Let $\chi$ be a multiplicative character. If $\chi\neq\eps$, the trivial multiplicative character, then $\sum_{t\in \mathbb{F}_{p}}\chi(t)=0$. Otherwise, the sum is $p$. 
\end{prop}
\begin{proof}
    The last assertion is as follows. Since $t$ runs through all elements of $\mathbb{F}_{p}$, we must have \begin{equation*}
        \sum_{t\in \mathbb{F}_{p}}\chi(t)=\sum_{t\in \mathbb{F}_{p}}\eps(t)=p.
    \end{equation*} 
    To prove the first assertion, we assume otherwise. Let there exist some $a\in \mathbb{F}_{p}^{\times}$ such that $\chi(a)\neq 1$, or $\chi$ does not map $a$ to $1$, hence $\chi$ is nontrivial. Let the desired sum be $T=\sum_{t\in \mathbb{F}_{p}}\chi(t)$. Then we may write
    \begin{equation*}
        \chi(a)T=\sum_{t\in \mathbb{F}_{p}}\chi(a)\chi(t)=\sum_{t\in \mathbb{F}_{p}}\chi(at).
    \end{equation*}
    This equates to $T$ itself as $at$ runs through the exact same number of elements from $\mathbb{F}_p$ as $t$ does, so $\chi(a)T=T$. Then $T(\chi(a)-1)=0$. We stated that necessarily $\chi(a)\neq 1$, so $T=0$, and we are finished.  
\end{proof}

We now turn our attention to Gauss and Jacobi sums. 

\subsection*{Gauss and Jacobi Sums}
We begin this technical subsection with a definition. 

\begin{defn}[Gauss sum]
    Let $\chi$ be some character on $\mathbb{F}_{p}$ and let $a\in\mathbb{F}_{p}$. Let 
    \begin{equation*}
        g_{a}(\chi)=\sum_{t\in\mathbb{F}_{p}}\chi(t)\zeta_{p}^{at},
    \end{equation*}
    where $\zeta_{p}=e^{2i\pi/p}$ is a $p$th root of unity. We say that $g_{a}(\chi)$ is a \textit{Gauss sum} on $\mathbb{F}_{p}$ belonging to the character $\chi$. 
\end{defn}

We will use $\zeta_{p}$ to always denote a $p$th root of unity from now on. By notational convention, when $a=1$, we write $g_{1}(\chi)=g(\chi)$. The next lemma provides us with a useful relationship between the Gauss sum and the character. 

\begin{lemma} \label{Lemma: two properties of the Gauss sum}
The following are true. 
    \begin{enumerate}
        \item If $a\neq 0$ and $\chi\neq \eps$, then $g_{a}(\chi)=\overline{\chi(a)}g_{1}(\chi)$. 
        \item If $a\neq 0$ and $\chi = \eps$, then $g_{a}(\eps)=0$.
    \end{enumerate}
\end{lemma}
\begin{proof}
    Let us begin by proving (1). Let $a\neq 0$ and $\chi\neq \eps$. Then 
    \begin{equation*}
        \chi(a)g_{a}(\chi) = \chi(a)\sum_{t\in\mathbb{F}_{p}}\chi(t)\zeta_{p}^{at} = \sum_{t\in\mathbb{F}_{p}}\chi(a)\chi(t)\zeta_{p}^{at} = \sum_{t\in\mathbb{F}_{p}}\chi(at)\zeta_{p}^{at} = g_{1}(\chi).
    \end{equation*}
    Then $\chi(a)g_{a}(\chi) = g_{1}(\chi)$, so $g_{a}(\chi)=g_{1}(\chi)\chi(a)^{-1} = \chi(a^{-1})g_{1}(\chi) = \overline{\chi(a)}g_{1}(\chi)$. 
    
    \indent We now prove (2). Let $a\neq 0$ but $\chi=\eps$. Since $\eps$ maps all $a\in\mathbb{F}_{p}$ to $1$, we have 
    \begin{equation*}
        g_{a}(\eps) = \sum_{t\in\mathbb{F}_{p}}\eps(t)\zeta_{p}^{at} = \sum_{t\in\mathbb{F}_{p}}\zeta_{p}^{at}.
    \end{equation*}
    Recall that $\mathbb{F}_{p}$ is the integers modulo $p$, so $t$ runs through all residue class representatives, or over $0\leq t\leq p-1$. Therefore $\sum_{t\in\mathbb{F}_{p}}\zeta_{p}^{at} = \sum_{t=0}^{p-1}\zeta_{p}^{at}$. Since $a\neq 0$, we consider two cases: (a) when $a\equiv 0 \pmod{p}$ and (b) when $a\not\equiv 0\pmod{p}$. Considering (a), if $a\equiv 0 \pmod{p}$, then for some $k\in\mathbb{Z}$, we have $\zeta_{p}^{a} = (e^{2i\pi/p})^{kp} = e^{2ki\pi}=1$ for all values of $k$. Then $\sum_{t=0}^{p-1}(\zeta_{p}^{a})^{t}=1+\cdots+1=p$. We now consider (b). If $a\not\equiv 0\pmod{p}$, then we can evaluate the sum as a finite geometric series. Then, re-indexing, 
    \begin{equation*}
        \sum_{t=0}^{p-1}\zeta_{p}^{at} = \sum_{t=1}^{p}\zeta_{p}^{at} = \frac{1(1-\zeta_{p}^{ap})}{1-\zeta_{p}^{a}} = \frac{\zeta_{p}^{ap}-1}{\zeta_{p}^{a}-1}.
    \end{equation*}
    We know that $\zeta_{p}^{ap}=1$ for all $p$ prime, so $\frac{\zeta_{p}^{ap}-1}{\zeta_{p}^{a}-1}=0/(\zeta_{p}^{a}-1)=0$. 
\end{proof}

In our proof of (2), we split the evaluation of the sum into two cases with dependence on the value of $a$. This result can be rewritten in the following form. 
\begin{lemma} \label{Lemma: result about zeta from special property prop about Gauss sums about zeta from special property prop about Gauss sums}
    \begin{equation*}
        \sum_{t=0}^{p-1}\zeta_{p}^{at} = \left\{ \begin{array}{ll}
            p, & a\equiv 0\pmod{p}, \\
            0, & a\not\equiv 0\pmod{p}.
        \end{array}
        \right.
    \end{equation*}
\end{lemma}

An easy corollary follows, where we denote the Kronecker delta with $\delta(x,y)$\footnote{Recall that we define $\delta(x,y)$ to be $1$ if $x=y$, and $0$ otherwise.}.

\begin{cor}[Corollary to Lemma \ref{Lemma: result about zeta from special property prop about Gauss sums about zeta from special property prop about Gauss sums}]\label{corollary to Lemma 3.7}
    \begin{equation*}
     p^{-1}\sum_{t=0}^{p-1}\zeta_{p}^{t(x-y)}  = \delta(x,y).
\end{equation*}
\end{cor}

The proof follows by evaluating cases when either $x\equiv y\pmod{p}$ or $x\not\equiv y\pmod{p}$. This is all we will need regarding the Gauss sum. 

\indent Now we define the Jacobi sum, which furthers the notion of a Gauss sum to include two characters. Gauss sums were briefly mentioned in \textit{Disquisitiones} by Gauss, but Jacobi sums only surfaced in 1827 when Carl Gustav Jacob Jacobi sent a letter to Gauss with his work. We begin with its definition. 
\begin{defn}[Jacobi sum]
    Let $\chi$ and $\lambda$ be two characters on $\mathbb{F}_{p}$. Then we define the \textit{Jacobi sum} over $\chi$ and $\lambda$ to be
    \begin{equation*}
        J(\chi,\lambda) = \sum_{\scriptsize\begin{aligned}a+b&=1 \\[-4pt] a,b &\in \mathbb{F}_{p}\end{aligned}}\chi(a)\lambda(b).
    \end{equation*}
\end{defn}

Jacobi sums are abstract in nature, so we provide an example of an elementary computation later in Example \ref{Example: computation of jacobi sum} after we have developed the necessary theory.

The defining property of Jacobi sums is their relationship with Gauss sums\footnote{Even so, the theory of Jacobi sums itself is rich, and this richness can be seen on a very high level in chapter 4 of \cite{lemmermeyer2000reciprocity}. A more modern approach to motivating Jacobi sums can be seen through determining the number of solutions to the Diophantine equation $x^{n}+y^{n}=1$, which is assessed with good rigor in chapter 8 of \cite{ireland1990classical}. A complete and rigorous treatment of Gauss and Jacobi sums with even more motivation can be found in \cite{berndt1998gauss}.}.

\begin{prop} \label{Proposition: basic properties of the jacobi sum and relation to Gauss sum}
    Let $\chi$ and $\lambda$ be characters on $\mathbb{F}_{p}$ such that neither is the trivial character $\eps$. If the composition $\chi\lambda\neq \eps$, then 
        \begin{equation*}
            J(\chi,\lambda) = \frac{g(\chi)g(\lambda)}{g(\chi\lambda)}.
        \end{equation*}
\end{prop}
\begin{proof}
    To begin, for $\zeta_{p}$ a $p$th root of unity, we have
    \begin{equation*}
        \label{eq:1} g(\chi)g(\lambda) = \bigg( \sum_{x}\chi(x)\zeta_{p}^{x} \bigg)\bigg( \sum_{y}\lambda(x)\zeta_{p}^{y} \bigg) = \sum_{x,y}\chi(x)\lambda(y)\zeta_{p}^{x+y} =  \sum_{t\in\mathbb{F}_{p}}\bigg(\sum_{x+y=t}\chi(x)\lambda(y)\bigg)\zeta_{p}^{t}.  
    \end{equation*}
    We consider two cases for the value of $t$. If $t=0$, then choosing to sum over $x$ and by the fact that the composition $\chi\lambda\neq \eps$, 
    \begin{equation*}
        \sum_{x+y=0}\chi(x)\lambda(y) = \sum_{x}\chi(x)\lambda(-x) = \sum_{x}\lambda(-1)\chi(x)\lambda(x) = \lambda(-1)\sum_{x}\chi\lambda(x) = 0
    \end{equation*}
    by Proposition \ref{Proposition: proof about sum of elements evaluated by a character prop}. In the case that $t\neq 0$, we define two new elements $x'$ and $y'$ as $x=tx'$ and $y=ty'$. Then, if we have $x+y=t$, then substituting we have $tx'+ty'=t$, so that $x'+y'=1$. Therefore 
    \begin{align*}
        \sum_{x+y=t}\chi(x)\lambda(y) & = \sum_{x'+y'=1}\chi(tx')\lambda(ty') = \sum_{x'+y'=1}\chi(t)\lambda(t)\chi(x')\lambda(y')  = \sum_{x'+y'=1}\chi\lambda(t)\chi(x')\lambda(y') \\ & = \chi\lambda(t)J(\chi,\lambda).
    \end{align*}
    If we substitute this into our evaluation of $g(\chi)g(\lambda)$, then we have 
    \begin{equation*}
        g(\chi)g(\lambda) = \sum_{t\in\mathbb{F}_{p}}\chi\lambda(t)J(\chi,\lambda)\zeta_{p}^{t} = J(\chi,\lambda)\sum_{t\in\mathbb{F}_{p}}\chi\lambda(t)\zeta_{p}^{t} = J(\chi,\lambda)g(\chi\lambda). 
    \end{equation*}
    Dividing both sides by $g(\chi\lambda)$, we thus have 
    \begin{equation*}
        J(\chi,\lambda) = \frac{g(\chi)g(\lambda)}{g(\chi\lambda)}.
    \end{equation*}
\end{proof}

The following result determines the value of the general Gauss sum, and we will then use it to determine the value of the Jacobi sum.

\begin{lemma}\label{Lemma: value of the general Gauss sum}
    If $\chi\neq \varepsilon$ is a nontrivial character, then $|g(\chi)|^{2}=p$. 
\end{lemma}
\begin{proof}
    The main idea for the proof is to evaluate the sum 
    \begin{equation*}
        \sum_{a\in\mathbb{F}_{p}}g_{a}(\chi)\overline{g_{a}(\chi)}
    \end{equation*}
    in two different ways and set the evaluations equal to one another. We will first evaluate the argument contained within the sum. Assume that $a\neq 0$. By (1) of Lemma \ref{Lemma: two properties of the Gauss sum}, we can write 
    \begin{equation*}
        \overline{g_{a}(\chi)} = \overline{\chi(a^{-1})g(\chi)} = \chi(a)\overline{g(\chi)}.
    \end{equation*}
    Taking the conjugate, we also have $g_{a}(\chi) = \chi(a^{-1})g(\chi)$. Multiplying and rearranging, we have
    \begin{equation*}
    \chi(a)\overline{g(\chi)}\chi(a^{-1})g(\chi) = \overline{g(\chi)}g(\chi) = |g(\chi)|^{2}.
    \end{equation*}
    Since $\sum_{a\in\mathbb{F}_{p}}$ sums over all elements of $\mathbb{F}_{p}$ except $a=0$, we consider this quantity $p-1$ times. So,
    \begin{equation*}
        \sum_{a\in\mathbb{F}_{p}}g_{a}(\chi)\overline{g_{a}(\chi)} = (p-1)|g(\chi)|^{2}.
    \end{equation*}
    Similarly, considering two parameters $x$ and $y$ and rewriting the argument as a double sum, we have
    \begin{equation*}
        g_{a}(\chi)\overline{g_{a}(\chi)} = \sum_{x\in\mathbb{F}_{p}}\sum_{y\in\mathbb{F}_{p}}\chi(x)\zeta_{p}^{ax}\overline{\chi(y)}\zeta_{p}^{ay} = \sum_{x\in\mathbb{F}_{p}}\sum_{y\in\mathbb{F}_{p}}\chi(x)\overline{\chi(y)}\zeta_{p}^{ax-ay}.
    \end{equation*}
    Summing over all elements of $\mathbb{F}_{p}$ and applying Corollary \ref{corollary to Lemma 3.7}, we have
    \begin{align*}  \sum_{a\in\mathbb{F}_{p}}\sum_{x\in\mathbb{F}_{p}}\sum_{y\in\mathbb{F}_{p}}\chi(x)\overline{\chi(y)}\zeta_{p}^{ax-ay} & = pp^{-1}\sum_{a\in\mathbb{F}_{p}}\sum_{x\in\mathbb{F}_{p}}\sum_{y\in\mathbb{F}_{p}}\chi(x)\overline{\chi(y)}\zeta_{p}^{ax-ay}p \\ & = p\sum_{x\in\mathbb{F}_{p}}\sum_{y\in\mathbb{F}_{p}}\chi(x)\overline{\chi(y)}\zeta_{p}^{ax-ay}\delta(x,y).
    \end{align*}
    If $x\not\equiv y\pmod{p}$ then the double sum will equate to $0$ as $\delta(x,y)$ will be $0$, and by Proposition \ref{Proposition: proof about sum of elements evaluated by a character prop} the sum of all multiplicative characters over $\mathbb{F}_{p}$ is 0. Therefore we consider when $x\equiv y\pmod{p}$. If this is true, then every term except when $x\equiv y\pmod{p}$ will be counted, leaving a total of $p-1$ terms. Since $x\equiv y\pmod{p}$, necessarily $\zeta_{p}^{ax-ay}=1$, so
    \begin{equation*}
        p\sum_{x\in\mathbb{F}_{p}}\sum_{y\in\mathbb{F}_{p}}\chi(x)\overline{\chi(y)}\zeta_{p}^{ax-ay}\delta(x,y) = p(p-1).
    \end{equation*}
    Equating our two evaluations, we have 
    \begin{align*}
        (p-1)|g(\chi)|^{2} & = p(p-1) \\
        |g(\chi)|^{2} & = p.
    \end{align*}
\end{proof}

Given the relation between the Gauss and Jacobi sum in Proposition \ref{Proposition: basic properties of the jacobi sum and relation to Gauss sum} and the value of the Gauss sum in Lemma \ref{Lemma: value of the general Gauss sum}, it is only natural that we ask what the value of the Jacobi sum is. 

\begin{cor}[Corollary to Proposition \ref{Proposition: basic properties of the jacobi sum and relation to Gauss sum}] \label{Corollary: value of the jacobi sum of two characters}
    Let $\chi$ and $\lambda$ be nontrivial multiplicative characters over $\mathbb{F}_{p}$. If their composition $\chi\lambda\neq\eps$, then $|J(\chi,\lambda)|=\sqrt{p}$. 
\end{cor}
\begin{proof}
    We apply Proposition \ref{Proposition: basic properties of the jacobi sum and relation to Gauss sum}. Take the absolute value of both sides to obtain 
    \begin{equation*}
        |J(\chi,\lambda)| = \bigg| \frac{g(\chi)g(\lambda)}{g(\chi\lambda)} \bigg| =
        \frac{|g(\chi)||g(\lambda)|}{|g(\chi\lambda)|}.
    \end{equation*}
    By Lemma \ref{Lemma: value of the general Gauss sum}, $g(\chi)=\sqrt{p}$ for any character $\chi$, so this is just $(\sqrt{p})^{2}/\sqrt{p}=\sqrt{p}$. 
\end{proof}

If we recall that the norm of an Eisenstein integer is $a^{2}-ab+b^{2}$, we will see a resemblance in the following result. It turns out that the uniqueness of the representation of the norm of an Eisenstein integer is crucial to proving the uniqueness of the representation of $p$ in Theorem \ref{Theorem: MAIN THEOREM}.

\begin{prop}\label{Proposition: exist integers a and b st a^2-ab+b^2}
    If $p\equiv 1\pmod{3}$, then there exist integers $a$ and $b$ such that $p=a^{2}-ab+b^{2}$. 
\end{prop}
\begin{proof}
    Since multiplicative characters form a cyclic group of order $p$ by our discussion in Remark \ref{Remark: character is a cyclic group}, there must be some generator element, say $\chi(a)$, that satisfies $\chi(a)^{p-1}=1$. Since $p\equiv 1\pmod{3}$, the order of the cyclic group we are dealing with must be a multiple of $3$. As a consequence of Lagrange's Theorem, there must exist some character of order $3$. 

    \indent Since there exists some character of order $3$, its values must be roots of the polynomial equation $x^{3}=1$, i.e., it must be a cube root of unity, taking on one of the values $1,\omega,$ and $\omega^{2}$. Therefore 
    \begin{equation*}
        J(\chi,\chi) = \sum_{u+v=1}\chi(u)\chi(v) = \sum_{u+v=1}\chi(uv)
    \end{equation*}
    must be an Eisenstein integer, and may be expressed in the form $J(\chi,\chi) = a+b\omega$, where $a,b\in\ZZ$. Recall that the norm of any Eisenstein integer in $\ZZ[\omega]$ is $N(a+b\omega) = a^{2}-ab+b^{2}$. Taking the absolute value of both sides and recalling Corollary \ref{Corollary: value of the jacobi sum of two characters}, we thus have
    \begin{align*}
        |a+b\omega| & = |J(\chi,\chi)| = \sqrt{p} \\
        N(a+b\omega) & = N(\sqrt{p}) \\
        a^{2}-ab+b^{2} & = (\sqrt{p})^{2} = p.
    \end{align*}
\end{proof}

\begin{ex}
    Suppose that $p=61\equiv 1\pmod{3}$. Then a possible pair for $a$ and $b$ is $(9,5)$, because $(9)^{2}-(9)(5)+(5)^{2} = 81-45+25 = 61$. 
\end{ex}

\begin{rem} \label{Remark: Implication of J(x,x)=a+bw}
    Regarding the Eisenstein integer $J(\chi,\chi)=a+b\omega$, an important result states that when $p\equiv 1\pmod{3}$, we have $a\equiv -1\pmod{3}$ and $b\equiv 0\pmod{3}$. The section on Jacobi sums in \cite{ireland1990classical} explains this result in necessary depth, as it is certainly not trivial and depends on some results about Jacobi sums that lie beyond the scope of this paper. We will use this fact in the proof of Theorem \ref{Theorem: MAIN THEOREM}.
\end{rem}

We are now prepared to prove the final result of this section. We will not present the proof for uniqueness as it becomes rather technical, but we give the more approachable proof of existence in favor of its implications. 

\begin{thm}\label{Theorem: exist unique integers st A^2+27B^2=4p}
    If $p\equiv 1\pmod{3}$, then there exist unique integers $A$ and $B$ that are determined up to sign such that $4p=A^{2}+27B^{2}$.
\end{thm}
\begin{proof}
    We want to manipulate $a^{2}-ab+b^{2}$ to be in a unique form. Proposition \ref{Proposition: exist integers a and b st a^2-ab+b^2} guarantees the existence of such $a$ and $b$. Notice that even if $a,b>0$, the representation is not unique for $p\equiv 1\pmod{3}$ because
    \begin{align*}
        a^{2}-ab+b^{2} & = (b-a)^{2}-(b-a)b+b^{2} \\ & = a^{2}-a(a-b)+(a-b)^{2},
    \end{align*}
    both of which are in the form $x^{2}-xy+y^{2}$ for $x,y\in\ZZ$. We will manipulate this expression such that it is unique. Recall that $p=a^{2}-ab+b^{2}$. Then 
    \begin{align}
        4p & = 4a^{2}-4ab+4b^{2} = (2a-b)^{2}+3b^{2} = (2b-a)^{2}+3a^{2} = (a+b)^{2}+3(a-b)^{2}. \label{Equation: (a+b)^2+3(a-b)^2=4p}
    \end{align}
    For this to be in the form that we want, we require $a-b$ to be a multiple of $3$, or $3|a-b$. Otherwise, either $3|a$ or $3|b$, but not both. Suppose that $3\nmid a$ and $3\nmid b$. We then consider two cases: (a) when $a\equiv 1\pmod{3}$ and $b\equiv 2\pmod{3}$ and (b) when $a\equiv 2\pmod{3}$ and $b\equiv 1\pmod{3}$. Let $k,l\in\ZZ$. For case (a), we have
    \begin{align*}
        (3k+1)^{2}-(3k+1)(3l+2)+(3l+2)^{2} & = 9k^{2}+6k-9kl-6k-3l+9l^{2}+12l+3 \\ & \equiv 0\pmod{3}.
    \end{align*}
    Case (b) follows similarly. Both cases show that $a^{2}-ab+b^{2}=p\equiv 0\pmod{3}$, which is impossible. Therefore $3|a-b$. Let $a-b=3B$ and $a+b=A$. Substituting this into Equation \eqref{Equation: (a+b)^2+3(a-b)^2=4p}, we obtain
    \begin{equation*}
        4p=A^{2}+3(3B)^{2} = A^{2}+27B^{2},
    \end{equation*}
    which is in the form that we wanted.

    As stated, the proof for uniqueness is rather long and  technical so we omit it. 
\end{proof}

\begin{ex}
    Suppose that $p=61$, so $61\equiv 1\pmod{3}$. Theorem \ref{Theorem: exist unique integers st A^2+27B^2=4p} asserts that there exist integers $A$ and $B$ such that $4(61)=244=A^{2}+27B^{2}$. By trial and error, we can see that when $A=1$ and when $B=3$, we have $244=1^{2}+27(3)^{2}=1+243=244$. This representation is unique up to sign, as $A=-1$ and $B=-3$ also satisfy the equality. 
\end{ex}

Theorem \ref{Theorem: exist unique integers st A^2+27B^2=4p} is also central to proving Euler's other conjectures for the cubic residuacity of small primes $p\equiv 1\pmod{3}$ in the form $a^{2}+3b^{2}$. As known by Gauss, Lagrange (who contributed to numerous results concerning quadratic residues using quadratic forms), and others, this result is rooted in the study of binary quadratic forms, an area that initially developed from Fermat's theorem on the sum of two squares and was studied rigorously by Gauss in \textit{Disquisitiones}\footnote{More of the richness regarding how binary quadratic forms paint a more complete picture of the development of both cubic and biquadratic residues can be found in section 4 of \cite{cox2013primes}.}.

\section*{The Law of Cubic Reciprocity}
While we cannot contain the full breadth of cubic reciprocity in this paper, it is an integral part of what we will use to evaluate the cubic character of 2, and only needs some more algebra. Cubic reciprocity occurs over the residue class ring $\ZZ[\omega]/\pi\ZZ[\omega]$, where $\pi$ is some prime element of $\ZZ[\omega]$. It turns out that this residue class ring is also a finite field with exactly $N\pi$ elements, and thus it retains properties that we need; for instance its multiplicative group $(\ZZ[\omega]/\pi\ZZ[\omega])^{\times}$ contains $N\pi-1$ elements, and by our discussion in Theorem \ref{Theorem: the multiplicative group of a finite field is cyclic} from earlier, it is cyclic. In summary,

\begin{thm} \label{Theorem: residue class ring contains Npi elements}
    The finite field $\ZZ[\omega]/\pi \ZZ[\omega]$ contains $N\pi$ elements. 
\end{thm}

This finite field is also unique in that it has a notion of congruence modulo a complex prime $\pi$, and the division algorithm applies. 

\indent One important classifying result for primes in $\ZZ[\omega]$ is that if the norm of some Eisenstein integer is a rational prime, then the Eisenstein integer must also be prime. To show this, we use contradiction. Suppose that $\pi$ is not prime in $\ZZ[\omega]$. Since it's not prime, without loss of generality, we can express it as a product of two non-unitary primes, say $\rho$ and $\gamma$, so $\pi=\rho\gamma$. If we take the norm of both sides, we have $N\pi=p=N\rho\gamma=N\rho N\gamma$. However, since $\rho$ and $\gamma$ are non-unitary, each of their norms must be greater than $1$. So, $N\rho N\gamma$ must also be greater than $1$, which is impossible since $p$ is a rational prime. Therefore $\pi$ is prime. 

\indent For example, let $\pi=3+\omega$, so $N(3+\omega)=3^{2}-3(1)+1=7$, which is prime. Therefore $3+\omega$ is prime and has no representation in terms of other primes. In summary, 

\begin{lemma}\label{Lemma: Npi=p implies pi is prime in Z[omega]}
    If $\pi\in\ZZ[\omega]$ has the property that its norm $N\pi = p$ a rational prime, then $\pi$ is prime in $\ZZ[\omega]$. 
\end{lemma}

An interesting implication of this result is that, by the definition of the norm and Proposition \ref{Proposition: exist integers a and b st a^2-ab+b^2}, the norm of each prime Eisenstein integer is uniquely expressible in terms of its components as a rational prime, or for some rational prime $p$ we have $N\pi = N(a+b\omega) = p = a^{2}-ab+b^{2}$. We will use this representation in our proof of Theorem \ref{Theorem: MAIN THEOREM}. 

\indent We give an example of the residue class ring $\ZZ[\omega]/\pi\ZZ[\omega]$. 

\begin{ex}
    Consider the prime $3+\omega$. We already showed that $3+\omega$ is prime. We are looking at the residue class ring $\ZZ[\omega]/(3+\omega)\ZZ[\omega]$. This ring contains $N(3+\omega)=7$ elements. It is possible to compute all 7 elements of this ring using brute force (i.e., considering all pairs $(a,b)$ where $a,b\in\{ 0,\ldots,2\}$), but we will not do this. Note also that in this definition of the residue class ring, the prime $3+\omega$ acts as a modulus, so only one coset representative modulo $3+\omega$ is in the ring. 
\end{ex}

If we recall our discussion in Footnote \ref{Footnote: mult group of int mod p is cyclic}, we can naturally suspect that there exists an analogous form to Fermat's Little Theorem over $\ZZ[\omega]/\pi\ZZ[\omega]$. Indeed, since $\ZZ[\omega]/\pi\ZZ[\omega]$ is a finite field, its multiplicative group is cyclic. So, for some generator element $\alpha\in\ZZ[\omega]/\pi\ZZ[\omega]$, all elements must satisfy 
\begin{equation}
    \alpha^{N\pi-1}\equiv 1\pmod{\pi} \label{Equation: Analogue to FLT in Z[w]/piZ[w]}
\end{equation}
for $\pi\nmid\alpha$. This can be referred to as the analogue to Fermat's Little Theorem over $\ZZ[\omega]/\pi\ZZ[\omega]$, and it is the fundamental equation for defining the cubic residue character. 

\indent Just as with quadratic characters, we need a corresponding cubic character. Note that since the quadratic character outputs solutions to the polynomial equation $x^{2}=1$, the cubic character outputs solutions to the polynomial equation $x^{3}=1$; namely, the cube roots of unity, $1,\omega$, and $\omega^{2}$. 
\begin{defn}[Cubic residue character]
    Let $N\pi\neq 3$. We say that the \textit{cubic residue character of $\alpha$ modulo $\pi$} is defined as
    \begin{enumerate}
        \item $(\alpha/\pi)_{3}=0$ if $\pi|\alpha$, 
        \item $\alpha^{\frac{N\pi-1}{3}}\equiv(\alpha/\pi)_{3} \pmod{\pi}$ where 
        \begin{equation*}
            (\alpha/\pi)_{3} = 
            \left\{
        \begin{array}{ll}
            1 & \text{if $\alpha$ is a cubic residue},\\
            \omega \ \text{or} \ \omega^{2} & \text{otherwise}.
        \end{array}
    \right.
        \end{equation*}
    \end{enumerate}
\end{defn}

With this defined, we can momentarily sidestep and provide an example computation of a Jacobi sum.

\begin{ex}\label{Example: computation of jacobi sum}
    Suppose that we are working with the Legendre symbol $(x/p)$ and the cubic residue character $(y/p)_{3}$ over $\mathbb{F}_{5}$. We compute the Legendre symbols first, so evaluate in the set $\ZZ/5\ZZ = \{ 0,1,2,3,4 \}$. By definition, $(0/5) = 0$. Since $1,4$ are quadratic residues, $(1/5) = (4/5) = 1$, and since $2,3$ are quadratic nonresidues, $(2/5) = (3/5) = -1$. For the cubic residue character, we evaluate in the cyclic group $\mathbb{F}_{5}^{\times} = \{1,2,3,4\}$. Notice that $2$ is the generator element of $\mathbb{F}_{5}^{\times}$, so we let $(2/5)_{3} = \omega$, and compute the remaining cubic characters accordingly. Then $(1/5)_{3} = 1, (3/5)_{3} = \omega^{2}$, and $(4/5)_{3} = \omega^{3}=1$. 
    
    Rewriting the condition that $a+b=1$ in the sum index, we compute the Jacobi sum as
    \begin{align*}
        J\left(\left( \frac{x}{5} \right),\left( \frac{y}{5} \right)_{3}\right) = \sum_{a=0}^{4}\left( \frac{a}{5} \right)\left( \frac{1-a}{5} \right)_{3} = (0)(1) + (1)(0) + (-1)(1) + (-1)(\omega^{2}) + (1)(\omega) = 2\omega.
    \end{align*}
\end{ex}

The cubic character behaves nearly exactly like the Legendre symbol. Most importantly, there is multiplicativity, so for $\alpha,\beta\in\ZZ[\omega]/\pi\ZZ[\omega]$, we have 
\begin{equation}
    (\alpha\beta/\pi)_{3} = (\alpha/\pi)_{3}(\beta/\pi)_{3}. \label{Equation: multiplicativity of cubic character}
\end{equation}

A proof follows by using Equation \eqref{Equation: Analogue to FLT in Z[w]/piZ[w]} and the definition of the cubic character. The most important definition over $\ZZ[\omega]/\pi\ZZ[\omega]$ is that of primary numbers. 

\begin{defn}[Primary number]
    Let $\pi\in\ZZ[\omega]$ be prime. We say that $\pi$ is \textit{primary} if $\pi\equiv \pm 1\pmod{3}$. 
\end{defn}

\begin{ex}
    Every element of $\ZZ[\omega]$ is representable in the form $a+b\omega$ for $a,b\in\ZZ$, so an element would be primary when its components satisfy $a\equiv 2\pmod{3}$ and $b\equiv 0\pmod{3}$. For instance, $8+6\omega$ is primary because $8\equiv 2\pmod{3}$ and $6\equiv 0\pmod{3}$, and $7+3\omega$ is primary because $7\equiv 2\pmod{3}$ and $3\equiv 0\pmod{3}$. However, $6+9\omega$ is not primary because both $6\equiv 9\equiv 0\pmod{3}$. 
\end{ex}

Eisenstein defined primary numbers to give a stronger notion of primality over $\ZZ[\omega]$. One can verify that the product of primary numbers remains primary, i.e., primary numbers are multiplicatively closed, and that the complex conjugate of a primary number remains primary. In fact, $\ZZ[\omega]$ is a UFD explicitly because there exists a unique primary factorization constructed from primary numbers. With these definitions, we would then be able to state the law of cubic reciprocity as in Theorem \ref{The Law of Cubic Reciprocity}. 

\indent An important discussion concerning cubic reciprocity and its relationship with biquadratic reciprocity and the formulation of the cubic and biquadratic character of 2 is now due. While cubic reciprocity wasn't formally published until 1844 by Eisenstein, both the cubic and biquadratic law were proven nearly a decade earlier by Jacobi in a sequence of lectures given in Königsberg from 1836-1837 after he had stated his theorems on cubic residuacity in 1827 without proof (see \cite[pg. 200]{lemmermeyer2000reciprocity}). Jacobi's lectures on number theory include complete proofs, but whether they were his or Eisenstein's still remains up to debate \cite{jacobi2007vorlesungen}. This nearly one decade gap between the 1836-1837 lectures and Eisenstein's 1844 formally published proof seems strange, but it is likely because the biquadratic law might have been easier than the cubic law: in a letter from Jacobi to Legendre in 1827, Jacobi wrote,

\begin{quote}
    Mr. Gauss presented to the Society of Göttingen two years ago a first memoir on the theory of biquadratic residues, which is much easier than the theory of cubic residues \cite{Jacobi1827-yn} (translated \cite[pg. 224]{lemmermeyer2000reciprocity}).
\end{quote}

Evidently, the biquadratic character of 2 was proven by Gauss in his first monograph of biquadratic reciprocity, which was published in 1828, far before his derivation of the cubic character of 2 \cite{gauss1876biquadraticmono}. Furthermore, Dirichlet (and Dedekind) derived the biquadratic character of 2 in an elementary fashion in their lectures on number theory \cite{dirichlet1863vorlesungen}.

\indent A computation applying cubic reciprocity is helpful to consider, though it should be noted that cubic characters are far more tedious to compute than Legendre symbols. 

\begin{ex}
    \footnote{This computation features a number of identities that we will not state here. For those who wish to follow carefully, statements and proofs of the identities can be found in chapter 9 section 4 of \cite{ireland1990classical} and section 3 of \cite{relyea2024finitefieldshigherreciprocity}.}Suppose we wish to determine the solvability of the congruence $x^{3}\equiv (3-\omega)\pmod{5}$. Note that this is valid because $N(3-\omega) = 13$ is prime. This amounts to evaluating the symbol $((3-\omega)/5)_{3}$. The modulus is already primary, but we need to make the argument primary. Recall that the units of $\ZZ[\omega]$ are $\pm 1,\pm\omega,\pm\omega^{2}$. We want to find some unit $u$ such that $u(3-\omega)\equiv 2\pmod{3}$. After some trial and error, we use Footnote \ref{Footnote: sum of cube roots of unity is 0} to find that $\omega^{2}(3-\omega) = 3\omega^{2}-1 = -4-3\omega \equiv 2\pmod{3}$. Therefore $((3-\omega)/5)_{3} = (\omega^{2}/5)^{-1}_{3}((-4-3\omega)/5)_{3}$. By applying Theorem \ref{The Law of Cubic Reciprocity} to the second character, this becomes $(\omega^{2}/5)^{-1}_{3}(5/(-4-3\omega))_{3}$. We want to reduce $5$ modulo $-4-3\omega$. Using rules of complex numbers, we see that $5/(-4-3\omega) = (-20-15\omega^{2})/13$. This fraction is approximately greater than $-1+\omega=-2-\omega^{2}$ using Footnote \ref{Footnote: sum of cube roots of unity is 0}, so we reduce $5$ by this multiple of $-4-\omega$. We obtain $5-(-4-3\omega)(-2-\omega^{2}) = 2\omega^{2}$. Evaluating the final symbol, we have
    \begin{align*}
        \left( \frac{3-\omega}{5}  \right)_{3} & = \left( \frac{\omega^{2}}{5}  \right)_{3}^{-1}\left( \frac{2\omega^{2}}{-4-3\omega} \right)_{3} = \left( \frac{\omega^{2}}{5} \right)_{3}^{-1}\left( \frac{\omega^{2}}{-4-3\omega}   \right)_{3}\left( \frac{2}{-4-3\omega} \right)_{3} 
        \\
        & = \left( \frac{\omega^{2}}{5} \right)_{3}^{-1}\left( \frac{\omega}{-4-3\omega}   \right)_{3}^{2}\left( \frac{-4-3\omega}{2} \right)_{3} = \left( \frac{\omega}{5} \right)_{3}^{-2}\left( \frac{\omega}{-4-3\omega} \right)_{3}^{2}\left( \frac{\omega}{2} \right)_{3} 
        \\
        & = (\omega^{\frac{N(5)-1}{3}})^{-2}(\omega^{\frac{N(-4-3\omega)-1}{3}})(\omega^{\frac{N(2)-1}{3}}) = (\omega^{2})^{-2}(\omega)(\omega)^{2} = \omega^{-1} = \omega^{2}.
    \end{align*}
    Therefore there is no solution to the congruence $x^{3}\equiv (3-\omega)\pmod{5}$. 
\end{ex}

\section*{The Cubic Character of 2}
Whereas the quadratic character of 2 determines the solvability of $x^{2}\equiv 2\pmod{p}$ for an odd prime $p$, the cubic character of 2 determines the solvability of $x^{3}\equiv 2\pmod{\pi}$, where $\pi\in\ZZ[\omega]$ is a prime. Given that cubic reciprocity occurs over both Eisenstein primes and rational primes, the cubic character of 2 can be considered for both prime and non-prime moduli. We will first take a look at a neat result for rational primes. 

\subsection*{A Neat Rational Case} 
In our first case, we suppose that the modulus $\pi$ is some rational prime $q$. Furthermore, since we deduced that $\pi$ must be primary, we let $q\equiv 2\pmod{3}$. This gives us the following generalization that, while not particularly helpful in considering the general case, is immensely powerful when considering the cubic character of only rational integers. 

\begin{prop}\label{Proposition about how every int is a cub res if q=2 mod 3}
    If $q\equiv 2\pmod{3}$ is a rational prime, then every integer is a cubic residue modulo $q$. 
\end{prop}
\begin{proof}
    We begin by assuming that $q\equiv 2\pmod{3}$ is a rational prime. Since $q$ is rational, we can work in the integers modulo $q$, namely $\mathbb{Z}/q\mathbb{Z}$. We define a group homomorphism $\phi:(\mathbb{Z}/q\mathbb{Z})^{\times} \rightarrow (\mathbb{Z}/q\mathbb{Z})^{\times}$ with the mapping $\phi(k)=k^{3}$ for some $k\in (\mathbb{Z}/q\mathbb{Z})^{\times}$. By the First Isomorphism Theorem, 
    \begin{equation*}
        (\mathbb{Z}/q\mathbb{Z})^{\times}/\text{Ker}(\phi)\approx \text{Im}(\phi).
    \end{equation*}
    We now determine the kernel of $\phi$. For an element $k$ to be contained within $\text{Ker}(\phi)$, the map induced by $\phi$ must yield $1$, the identity of $(\mathbb{Z}/q\mathbb{Z})^{\times}$. In other words, $k^{3}=1$ iff $k\in \text{Ker}(\phi)$. However, Footnote \ref{Footnote: mult group of int mod p is cyclic} asserts that the multiplicative group $(\mathbb{Z}/q\mathbb{Z})^{\times}$ is cyclic and has order $q-1$. But, $3\nmid q-1$, so naturally the relation $k^{3}=1$ is possible if and only if $k=1$, as it maps to the identity. Thus $\text{Ker}(\phi)$ is trivial, in that it only contains one element, so that 
    \begin{equation*}
        |\text{Im}(\phi)|=|(\mathbb{Z}/q\mathbb{Z})^{\times}/\text{Ker}(\phi)|=|(\mathbb{Z}/q\mathbb{Z})^{\times}|/1 = |(\mathbb{Z}/q\mathbb{Z})^{\times}|.
    \end{equation*}
    This satisfies the condition for $\phi$ to be surjective, so due to the mapping $\phi$ defined earlier, every element of $(\mathbb{Z}/q\mathbb{Z})^{\times}$ is a perfect cube, i.e., every integer is a cubic residue modulo $q$.
\end{proof}

The implications of this result are beautifully simple. We look at an example. 
\begin{ex}
    Suppose that the modulus is 11, because $11\equiv 2\pmod{3}$. Proposition \ref{Proposition about how every int is a cub res if q=2 mod 3} tells us that every integer must be a cubic residue modulo $11$. Since we are dealing modulo $11$, every least residue $0,1,\ldots,10$ must be a cubic residue. We have

\begin{center}
\begin{tabular}{c c c c}
 $ 0^3 \equiv 0 \pmod{11}$ &
$1^3 \equiv 1 \pmod{11}$ &
$2^3 \equiv 8 \pmod{11}$ &
$3^3 \equiv 5 \pmod{11}$ 
\\
$4^3 \equiv 9 \pmod{11}$ &
$5^3 \equiv 4 \pmod{11}$ &
$6^3 \equiv 7 \pmod{11}$ &
$7^3 \equiv 2 \pmod{11}$ 
\\
$8^3 \equiv 6 \pmod{11}$ &
$9^3 \equiv 3 \pmod{11}$ &
$10^3 \equiv 10 \pmod{11}$. 
\end{tabular}
\end{center}

\noindent Notice that each integer in the finite field $\ZZ/11\ZZ$ is represented, so every element is a cubic residue modulo $11$, and we get that $(2/q)_{3}=1$ for free whenever $q\equiv 2\pmod{3}$. 
\end{ex}

\subsection*{The Complex Cases} 
The remaining cases are more difficult. Since Proposition \ref{Proposition about how every int is a cub res if q=2 mod 3} shows us that rational primes are always cubic residues of 2 modulo some rational prime $q\equiv 2\pmod{3}$, we consider the case where the modulus is a complex prime of the form $a+b\omega$. 

\indent The first result we will prove eliminates consideration of most complex elements in the final result. The first important step is noticing that the congruence $x^{3}\equiv 2\pmod{\pi}$, where $\pi\in\ZZ[\omega]$ is prime, is solvable if and only if every congruence $x^{3}\equiv 2\pmod{\pi^{\prime}}$ is also solvable, where $\pi^{\prime}$ denotes an associate\footnote{Recall that in a commutative ring $R$, two elements $r,s\in R$ are \textit{associate} if there is some unitary element $u\in R$ such that $us=r$. In this case, we are asserting that there exists some $u\in \ZZ[\omega]$ such that $\pi = u\pi^{\prime}$.\label{Footnote: associate congruences}} of $\pi$. To differentiate between $\pi$ and its associates, we can assume that $\pi$ is primary. 
\begin{lemma} \label{Lemma: pi=1 mod 2 for x^3=2 mod pi to be solvable}
    The cubic congruence $x^{3}\equiv 2\pmod{\pi}$, where $\pi=a+b\omega\in \ZZ[\omega]$ is prime for $a,b\in\ZZ$, is solvable if and only if $\pi\equiv 1\pmod{2}$, i.e., if and only if $a\equiv 1\pmod{2}$ and $b\equiv 0\pmod{2}$. 
\end{lemma}
\begin{proof}
    \indent Let $\pi=a+b\omega$ be a primary prime. By Theorem \ref{The Law of Cubic Reciprocity}, we have that $(2/\pi)_{3} = (\pi/2)_{3}$. We evaluate the right-hand-side of this equality. By the definition of the cubic character, 
    \begin{equation*}
        \pi^{\frac{N(2)-1}{3}} = \pi^{(4-1)/3} = \pi \equiv (\pi/2)_{3}\pmod{2}. 
    \end{equation*}
    Notice that $(\pi/2)_{3}=1$, i.e., $\pi$ is a cubic residue modulo $2$, if and only if the congruence $x^{3}\equiv \pi\pmod{2}$ is solvable. This congruence is solvable, however, if and only if $\pi\equiv 1\pmod{2}$, because $\pi$ would no longer be prime if $\pi\equiv 0\pmod{2}$. Therefore $(\pi/2)_{3}=1$ if and only if $\pi\equiv 1\pmod{2}$. Similarly, we have that $(2/\pi)_{3}=1$ if and only if $\pi\equiv 1\pmod{2}$. In either case, we thus have that the congruence $x^{3}\equiv 2\pmod{\pi}$ is solvable if and only if $\pi\equiv 1\pmod{2}$. 
\end{proof}

The importance of Lemma \ref{Lemma: pi=1 mod 2 for x^3=2 mod pi to be solvable} is that it reduces the problem to only considering some rational cases, so it is fundamental to proving Theorem \ref{Theorem: MAIN THEOREM}. We are now in the final stretch. Euler's conjectures regarding the cubic residuacity of special integers only appeared posthumously in his \textit{Tractatus} \cite{euler1849tractatus}, which was published in 1849 despite the fact that the contents were originally written between 1748 and 1750. The history of who proved the rule first is complex as Gauss produced many relevant informal notes that were only published posthumously, but other mathematicians also published proofs throughout Gauss' lifetime. However, Gauss did manage to evaluate the cubic and biquadratic characters of 2 in his sketches \cite{gauss1900werke8article} using the incredibly important ideas from Theorem \ref{Theorem: exist unique integers st A^2+27B^2=4p}, just as we do now.

\begin{proof}
[Proof of Theorem \ref{Theorem: MAIN THEOREM}]
    Let $\pi=a+b\omega$ be prime. By the definition of the norm, $N\pi=p=a^{2}-ab+b^{2}$, where $a,b\in\ZZ$. We start by proving the forward direction by supposing that the cubic congruence $x^{3}\equiv 2\pmod{p}$ is solvable. Since $\pi$ is associate to $p$, we know that $x^{3}\equiv 2\pmod{\pi}$ must also be solvable. So, by Lemma \ref{Lemma: pi=1 mod 2 for x^3=2 mod pi to be solvable}, $\pi\equiv 1\pmod{2}$. Notice that this implies that in the expression $a+b\omega$, $a$ must be odd and $b$ must be even.

    \indent Let us now examine $p=a^{2}-ab+b^{2}$. By Remark \ref{Remark: Implication of J(x,x)=a+bw} and Proposition \ref{Proposition: exist integers a and b st a^2-ab+b^2}, we have $b\equiv 0\pmod{3}$, or $3|b$, for the representation $N(a+b\omega) = a^{2}-ab+b^{2} = p$. Multiplying both sides by 4, we obtain $4p=4a^{2}-4ab+4b^{2}=4a^{2}-4ab+b^{2}+3b^{2}$. Notice that we can factor the right-hand-side, yielding $4p=(2a-b)^{2}+3b^{2}$. Now let $A=2a-b$ and $B=b/3$. Plugging these in, we obtain
    \begin{equation*}
        4p = A^{2}+3(3B)^{2} = A^{2}+27B^{2}.
    \end{equation*}
    By Theorem \ref{Theorem: exist unique integers st A^2+27B^2=4p}, this representation is unique, in that both $A$ and $B$ are unique integers up to sign. Notice that since $b$ is even and $3|b$, for $m,n\in\ZZ$, we write $b=2m\cdot 3n$. Therefore $B=6mn/3=2mn$, which is even. Since $4p = A^{2}+27B^{2}$, we also know that $A^{2}+27B^{2}$ must be even, which further requires $A$ to be even. Substitute $C=A/2$ and $D=B/2$, which are both necessarily integers. Therefore $p=C^{2}+27D^{2}$, proving the forward direction.

    \indent We now prove the backward direction, undoing what we did in the forward direction. Suppose that there exist $C,D\in\ZZ$ such that $p=C^{2}+27D^{2}$. Multiplying both sides by $4$, we have
    \begin{equation*}
        4p = 4C^{2}+4\cdot 27D^{2} = (2C)^{2}+27(2D)^{2}.
    \end{equation*}
    Again, by Theorem \ref{Theorem: exist unique integers st A^2+27B^2=4p}, there must exist some unique integer $B=\pm 2D$, which implies that $B$ is even. Consequently $b$ must also be even, implying that $b\equiv 0\pmod{2}$; necessarily, $\pi\equiv 1\pmod{2}$ since $\pi$ is prime. Therefore by the backward direction of Lemma \ref{Lemma: pi=1 mod 2 for x^3=2 mod pi to be solvable}, the cubic congruence $x^{3}\equiv 2\pmod{\pi}$ is solvable. We now need to show that this extends further to modulus a rational prime $p$. 

    \indent Theorem \ref{Theorem: residue class ring contains Npi elements} shows that the residue class ring $\ZZ[\omega]/\pi\ZZ[\omega]$ contains exactly $N\pi=p$ elements. Therefore, among all elements of $\ZZ[\omega]/\pi\ZZ[\omega]$, there must exist some $k\in\ZZ$ such that $k^{3}\equiv 2\pmod{\pi}$. By definition of congruence, this means $\pi|(k^{3}-2)$. By our discussion preceding Footnote \ref{Footnote: associate congruences}, it must also be true that $\pi^{\prime}|(k^{3}-2)$, where $\pi^{\prime}$ is some associate of $\pi$. Since $\pi\pi^{\prime}=p$, we have $p|(k^{3}-2)^{2}$, so $p|(k^{3}-2)$. Rewriting, this implies that $k^{3}\equiv 2\pmod{p}$, or $k$ is a solution to the cubic congruence $x^{3}\equiv 2\pmod{p}$. 
\end{proof}

Notice that this result is only applicable when considering moduli greater than or equal to 28, since $C^{2}+27D^{2}\geq 28$ when $|C|,|D|\geq 1$.  
\begin{ex}
    In this first example we examine the case where the modulus is 29. Notice that there is no ordered pair $(A,B)$ such that $29=A^{2}+27B^{2}$. Therefore there is no solution to the congruence $x^{3}\equiv 2\pmod{29}$. 
\end{ex}
\begin{ex}
    In this second example we let the modulus be 259. Since $259=4^{2}+27(3^{2})$, there must exist a solution to the cubic congruence $x^{3}\equiv 2\pmod{259}$, so $(2/259)_{3}=1$. 
\end{ex}

\subsection*{Acknowledgements} The author wishes to acknowledge the late Dr. Tamar Avineri who taught at the North Carolina School of Science and Mathematics for over $20$ years. Dr. Avineri was always endlessly supportive, inspiring, and so full of life, insight, and virtue: she could not have been a better advisor, always ensuring that every proof was sufficiently parsimonious. The author also wishes to thank the anonymous referees for their comments and suggestions.
\bibliographystyle{unsrt}
\bibliography{biblio.bib}

\begin{thebibliography}{10}

\bibitem{lemmermeyer2000reciprocity}
F.~Lemmermeyer.
\newblock {\em Reciprocity Laws: From Euler to Eisenstein}.
\newblock Springer Monographs in Mathematics. Springer Berlin Heidelberg, 2000.

\bibitem{Gauß_1986}
Carl~Friedrich Gauß.
\newblock {\em Disquisitiones Arithmeticae: Engl. Ed.}
\newblock Springer, 1986.

\bibitem{biblioteca1863werke2}
Biblioteca Provinciale and Reale~Biblioteca di~Marina.
\newblock {\em Werke Carl Friedrich Gauss: 2}.
\newblock Königlichen Gesellschaft der Wissenschaften, 1863.

\bibitem{cox2013primes}
D.A. Cox.
\newblock {\em Primes of the Form x2+ny2: Fermat, Class Field Theory, and Complex Multiplication}.
\newblock Pure and Applied Mathematics: A Wiley Series of Texts, Monographs and Tracts. Wiley, 2013.

\bibitem{relyea2024finitefieldshigherreciprocity}
Matias~Carl Relyea.
\newblock {\em On Finite Fields and Higher Reciprocity}.
\newblock arXiv, 2024.

\bibitem{ireland1990classical}
Kenneth Ireland and Michael Rosen.
\newblock {\em A Classical Introduction to Modern Number Theory}, volume~84.
\newblock Springer Science \& Business Media, 1990.

\bibitem{gauss1900werke8}
Carl~Friedrich Gauß.
\newblock {\em Werke: Vol. 8}, volume~8.
\newblock Teubner, 1900.

\bibitem{collison1977cubbi}
Mary~Joan Collison.
\newblock The origins of the cubic and biquadratic reciprocity laws.
\newblock {\em Archive for History of Exact Sciences}, 17(1):63--69, 1977.

\bibitem{Gray_2018}
Jeremy Gray.
\newblock {\em A history of abstract algebra: From algebraic equations to modern algebra}.
\newblock Springer, 2018.

\bibitem{berndt1998gauss}
B.C. Berndt, R.J. Evans, and K.S. Williams.
\newblock {\em Gauss and Jacobi Sums}.
\newblock Wiley-Interscience and Canadian Mathematics Series of Monographs and Texts. Wiley, 1998.

\bibitem{jacobi2007vorlesungen}
C.G.J. Jacobi, F.~Lemmermeyer, and H.~Pieper.
\newblock {\em Vorlesungen {\"u}ber Zahlentheorie: Wintersemester 1836/37, K{\"o}nigsberg}.
\newblock Rauner, 2007.

\bibitem{Jacobi1827-yn}
C~G~J Jacobi.
\newblock {\em Letter to Legendre of 5}.
\newblock Borchardt, C W and Chelsea, Ed, New York; Berlin, 1827.

\bibitem{gauss1876biquadraticmono}
Carl~Friedrich Gauß.
\newblock Theoria residuorum biquadraticorum -commentatio prima.
\newblock {\em Werke (Göttingen: König- lichen Gesellschaft der Wissenschaften, 1876)}, pages 65--92, 1876.

\bibitem{dirichlet1863vorlesungen}
P.G.L. Dirichlet and R.~Dedekind.
\newblock {\em Vorlesungen {\"u}ber zahlentheorie}.
\newblock F. Vieweg und sohn, 1863.

\bibitem{euler1849tractatus}
Leonhard Euler.
\newblock {\em Tractatus de numerorum doctrina capita XVI, quae supersunt}.
\newblock Euler Archive - All Works, 1849.

\bibitem{gauss1900werke8article}
Carl~Friedrich Gauß.
\newblock Notizen über cubische und biquadratische reste.
\newblock {\em Nachlaß}, 1900.

\end{thebibliography}

\

\noindent \textbf{Matias C. Relyea} recently graduated from the North Carolina School of Science and Mathematics. Currently he is most interested in number theory, but he hopes to widen his knowledge of other areas of mathematics. In his free time, Matias enjoys playing the violin, badminton, and competing casually at TeXnique.
\end{document}